\documentclass[11pt]{amsart}

\usepackage{amsmath,amssymb,mathtools}
\usepackage{booktabs}
\usepackage{microtype}
\usepackage[hidelinks]{hyperref}
\hypersetup{
  pdftitle={Principal nonsingularity of the Fourier matrices of orders 70 and 143},
  pdfkeywords={Fourier matrix, principal minors, square-free order, finite fields}
}

\newtheorem{theorem}{Theorem}[section]
\newtheorem{proposition}[theorem]{Proposition}
\newtheorem{lemma}[theorem]{Lemma}
\theoremstyle{remark}
\newtheorem{remark}[theorem]{Remark}

\newcommand{\F}{\mathcal{F}}
\newcommand{\Z}{\mathbb{Z}}

\title[Fourier matrices of orders 70 and 143]{Principal nonsingularity of the Fourier matrices of orders \(70\) and \(143\)}

\author{Jian Gu} 
\address{Department of Mathematics, Fudan University}

\author{Liyi Zhou} 
\address{Department of Mathematics, Fudan University}

\author{Yuhu Wang} 
\address{Department of Mathematics, Fudan University}

\date{August 2026}


\begin{document}

\begin{abstract}
We give computer-assisted proofs that every principal minor of each of
the \(70\times70\) and \(143\times143\) Fourier matrices is nonzero.
A lifting theorem of Caragea, Lee, Malikiosis, and Pfander reduces the
two assertions to the nonvanishing of all principal minors of the
Fourier matrix of order \(10\) in characteristic \(7\), and of order
\(11\) in characteristic \(13\), respectively.  We realize primitive
roots in \(\mathbb F_{7^4}\) and \(\mathbb F_{13^{10}}\) and evaluate
all \(2^{10}\) and \(2^{11}\) principal determinants by exact,
division-free arithmetic.  None vanishes.  The lifting theorem in fact
yields the stronger conclusions that every \(10\)-principal minor of
the order-\(70\) matrix and every \(11\)-principal minor of the
order-\(143\) matrix is nonzero.  Self-contained standard-library
verifiers for the finite-field calculations accompany the paper.
\end{abstract}

\maketitle

\section{Introduction}

For \(N\geq2\), let
\[
  \F_N=(\omega_N^{ij})_{i,j\in\Z_N},
  \qquad
  \omega_N=e^{-2\pi i/N}.
\]
We adopt the convention
\(\det\F_N[\varnothing,\varnothing]=1\).
A classical theorem attributed to Chebotar\"ev states that when \(N=p\)
is prime, every square minor of \(\F_p\) is nonzero.  This theorem is
closely related to generalized Vandermonde determinants: its historical
antecedents include Mitchell's determinant identities, and useful modern
proofs were given by Evans and Isaacs, Frenkel, and Tao
\cite{Mitchell1881,EvansIsaacs1976,Frenkel2004,Tao2005}.  Conversely, if
\(N\) is composite, suitable two-by-two Fourier submatrices are singular.
Thus the all-minors property characterizes prime order.

The prime-order theorem is also a sharp finite uncertainty principle.
For a nonzero function \(f\) on \(\Z_p\), it implies
\[
  |\operatorname{supp}f|+|\operatorname{supp}\widehat f|\geq p+1,
\]
as emphasized by Tao \cite{Tao2005}.  This sharpens the product-type
uncertainty inequality of Donoho and Stark in the prime cyclic setting;
for general finite abelian groups, the possible additive bounds reflect
the subgroup and divisor structure \cite{DonohoStark1989,Meshulam2006}.
Fourier submatrices have consequently been studied from several nearby
directions.  Delvaux and Van Barel studied the structure of
rank-deficient Fourier submatrices \cite{DelvauxVanBarel2008}, while
Barnett quantified the severe ill-conditioning that can occur even for
nonsingular contiguous Fourier submatrices \cite{Barnett2022}.  These
results also explain why floating-point searches alone are poorly suited
to certifying exact nonvanishing.

Invertibility of Fourier submatrices enters the construction and
recombination of exponential Riesz bases.  Relevant developments include
the combination theorem of Kozma and Nitzan
\cite{KozmaNitzan2015}, structured exponential bases of Caragea and Lee
\cite{CarageaLee2022}, and bases with restricted supports studied by Lee,
Pfander, and Walnut \cite{LeePfanderWalnut2023}.  In the setting of woven
Riesz bases, Cabrelli, Molter, and Negreira explicitly isolated the
importance of principal Fourier minors \cite{CabrelliMolterNegreira2025}.

Requiring only \emph{principal} minors leads to a different arithmetic
boundary from the full Chebotar\"ev property.  For \(N\geq4\), Caragea
and Lee proved that square-freeness is exactly the condition for the nonvanishing of all
two-by-two and three-by-three principal minors, and that nonsquare-free
orders have zero principal minors of every intermediate size
\cite{CarageaLee2024}.  Together with the Riesz-basis formulation above,
this led to the conjecture
\[
  \F_N\text{ has no zero principal minor}
  \quad\Longleftrightarrow\quad
  N\text{ is square-free};
\]
see \cite{CabrelliMolterNegreira2025,CarageaLee2024,
CarageaLeeMalikiosisPfander2025}.

Recent progress combines cyclotomic algebra with reduction to finite
characteristic.  Zhang established a finite-field analogue of
Chebotar\"ev's theorem under explicit hypotheses
\cite{Zhang2019}; Loukaki used this circle of ideas to treat products of
two primes under an order and size condition \cite{Loukaki2025}; and
Emmrich and Kunis obtained a real analogue and an improved finite-field
version of Chebotar\"ev's theorem \cite{EmmrichKunis2025}.  Caragea,
Lee, Malikiosis, and Pfander developed
a lifting theorem for congruence-balanced minors and proved the conjecture
for the families \(2p,3p,5p,6p,7p\), among others
\cite{CarageaLeeMalikiosisPfander2025}.  Zhou proved principal
nonsingularity for \(\Z_p\times\Z_q\) when \(p,q\) are distinct odd
primes, \(q\) is sufficiently large, and \(q\) generates
\(\Z_p^*\); he also proved principal nonsingularity upon a column
permutation for \(\Z_2^k\times\Z_q\) when \(q\) is odd and
\(k\in\{1,2,3\}\) \cite{Zhou2026}.

The lifting framework of
\cite{CarageaLeeMalikiosisPfander2025} identifies
\[
  143=11\cdot13
  \qquad\text{and}\qquad
  70=2\cdot5\cdot7
\]
as, respectively, the smallest two-prime and three-prime orders left
unresolved by their results
\cite[Theorem~1.3 and the discussion following it]
{CarageaLeeMalikiosisPfander2025}.  To the best of our knowledge, our
main result settles both cases.

\begin{theorem}\label{thm:main}
For \(N\in\{70,143\}\) and every \(A\subseteq\Z_N\),
\[
  \det \F_N[A,A]\neq0.
\]
\end{theorem}

The point is that one need not examine the \(2^{70}\) or \(2^{143}\)
principal submatrices of the two matrices.  The lifting theorem reduces
the problems to exact calculations involving the \(2^{10}=1024\)
principal submatrices of \(\F_{10}\) over \(\mathbb F_{7^4}\) and the
\(2^{11}=2048\) principal submatrices of \(\F_{11}\) over
\(\mathbb F_{13^{10}}\).  The calculations are small, use no
floating-point arithmetic, and are reproduced in self-contained Python
files included with the source.

From a broader perspective, the problem belongs to a family of discrete
rigidity questions in which linear relations prevent nonzero
configurations from being too sparse.  Related support and
unique-continuation phenomena arise for discrete Schr\"odinger equations
\cite{LiAnderson2D2022,LiSimplex2026,LiSupport2026,LiZhang2022}.
These works provide conceptual motivation, but their quantitative
estimates are not logical inputs to the proof below, which instead uses
an exact finite-field computation together with the Fourier-minor
lifting theorem.

\section{Reduction to characteristics 7 and 13}

We recall the congruence-balanced minors used in the lifting theorem.
Let \(d\mid N\), and let \(R,C\subseteq\Z_N\) have the same cardinality.
The minor \(\F_N[R,C]\) is called \(d\)-principal if
\[
  \#\{r\in R:r\equiv a\pmod d\}
  =
  \#\{c\in C:c\equiv a\pmod d\}
  \qquad (a\in\Z_d).
\]
An ordinary principal minor has \(R=C\), so it is \(d\)-principal for
every divisor \(d\) of \(N\).

We use the following norm-form consequence of the lifting theorem.

\begin{proposition}[Finite-characteristic lifting
{\cite[Theorem~1.4 and Section~3.2]
{CarageaLeeMalikiosisPfander2025}}]
\label{prop:lifting}
Let \(N=pN'\) be square-free, where \(p\) is prime.  For
\(S\subseteq\Z_{N'}\), set
\[
  \Delta_S=
  \det(\omega_{N'}^{ij})_{i,j\in S}
  \in\mathbb Z[\omega_{N'}].
\]
If
\[
  p\nmid
  N_{\mathbb Q(\omega_{N'})/\mathbb Q}(\Delta_S)
  \qquad\text{for every }S\subseteq\Z_{N'},
\]
then every \(N'\)-principal minor of the complex matrix \(\F_N\) is
nonzero.
\end{proposition}

\subsection{The order-70 reduction}

We now describe precisely the characteristic-\(7\) matrix that occurs
when \(N'=10\).  Put
\[
  \Phi_{10}(X)=X^4-X^3+X^2-X+1
\]
and
\[
  K=\mathbb F_7[X]/\bigl(\Phi_{10}(X)\bigr),
  \qquad t=X\bmod\Phi_{10}(X).
\]

\begin{lemma}\label{lem:field70}
The polynomial \(\Phi_{10}\) is irreducible over \(\mathbb F_7\);
therefore \(K\cong\mathbb F_{7^4}\).  The class \(t\in K\) has exact
order \(10\), and
\[
  W=(t^{ij})_{0\leq i,j<10}
\]
is the reduction of the order-\(10\) Fourier matrix modulo the unique
prime above \(7\).
\end{lemma}

\begin{proof}
The powers of \(7\) modulo \(10\) are
\[
  7,\ 9,\ 3,\ 1,
\]
so
\[
  \operatorname{ord}_{10}(7)=4=\varphi(10).
\]
The standard factorization law for cyclotomic polynomials over finite
fields implies that \(\Phi_{10}\) remains irreducible modulo \(7\).
Since
\[
  (X+1)\Phi_{10}(X)=X^5+1,
\]
we have \(t^5=-1\) and hence \(t^{10}=1\).  Moreover, irreducibility
implies \(t\neq-1\).  Thus \(t\) cannot have order \(1\), \(2\), or
\(5\), and consequently has exact order \(10\).

Equivalently, if \(L=\mathbb Q(\omega_{10})\) and
\(\mathcal O_L=\mathbb Z[\omega_{10}]\), then \(7\) is inert in \(L\)
and
\[
  \mathcal O_L/(7)\cong K.
\]
There is thus a unique prime of \(\mathcal O_L\) above \(7\).
For \(S\subseteq\Z_{10}\), put
\[
  \Delta_S=
  \det(\omega_{10}^{ij})_{i,j\in S}
  \in\mathcal O_L.
\]
Under the quotient map
\(\mathcal O_L\to\mathcal O_L/(7)\cong K\), the element \(\Delta_S\)
reduces to
\[
  D_S(t)=\det(t^{ij})_{i,j\in S},
\]
after choosing the displayed identification.  A different identification
sends \(t\) to one of
\[
  t,\quad t^7,\quad t^{7^2}=t^9,\quad t^{7^3}=t^3.
\]
These elements form a single Frobenius orbit, and
\[
  D_S(t^{7^r})=D_S(t)^{7^r}.
\]
Thus nonvanishing is independent of this choice.

Because \((7)\) is the unique prime of \(\mathcal O_L\) above \(7\),
\[
  D_S(t)\neq0
  \quad\Longleftrightarrow\quad
  \Delta_S\notin(7)
  \quad\Longleftrightarrow\quad
  7\nmid N_{L/\mathbb Q}(\Delta_S).
\]
Hence nonvanishing of every \(D_S(t)\) verifies precisely the norm
hypothesis of Proposition~\ref{prop:lifting}.
\end{proof}

\begin{proposition}\label{prop:reduction70}
If
\[
  \det W[S,S]\neq0
  \qquad\text{for every }S\subseteq\Z_{10},
\]
then every \(10\)-principal minor of \(\F_{70}\) is nonzero.
\end{proposition}

\begin{proof}
By Lemma~\ref{lem:field70}, the assumed finite-field nonvanishing implies
\[
  7\nmid
  N_{\mathbb Q(\omega_{10})/\mathbb Q}
  \left(\det(\omega_{10}^{ij})_{i,j\in S}\right)
\]
for every \(S\subseteq\Z_{10}\).  Proposition~\ref{prop:lifting}
therefore applies with \((N,p,N')=(70,7,10)\).
\end{proof}

\subsection{The order-143 reduction}

For the second case, put
\[
  \Phi_{11}(X)=1+X+\cdots+X^{10}
\]
and
\[
  K_{13}=\mathbb F_{13}[X]/\bigl(\Phi_{11}(X)\bigr),
  \qquad u=X\bmod\Phi_{11}(X).
\]

\begin{lemma}\label{lem:field143}
The polynomial \(\Phi_{11}\) is irreducible over \(\mathbb F_{13}\);
therefore \(K_{13}\cong\mathbb F_{13^{10}}\).  The class \(u\in K_{13}\)
has exact order \(11\), and
\[
  W_{11}=(u^{ij})_{0\leq i,j<11}
\]
is the reduction of the order-\(11\) Fourier matrix modulo the unique
prime above \(13\).
\end{lemma}

\begin{proof}
Since \(13\equiv2\pmod{11}\), direct calculation gives
\[
  \operatorname{ord}_{11}(13)=\operatorname{ord}_{11}(2)
  =10=\varphi(11).
\]
The cyclotomic factorization law over finite fields therefore shows that
\(\Phi_{11}\) remains irreducible modulo \(13\).  The identity
\[
  (X-1)\Phi_{11}(X)=X^{11}-1
\]
gives \(u^{11}=1\), while
\(\Phi_{11}(1)=11\neq0\) in \(\mathbb F_{13}\), so \(u\neq1\).
As \(11\) is prime, \(u\) has exact order \(11\).

Equivalently, if \(L_{11}=\mathbb Q(\omega_{11})\), then \(13\) is
inert in \(L_{11}\) and
\[
  \mathbb Z[\omega_{11}]/(13)\cong K_{13}.
\]
For \(S\subseteq\mathbb Z_{11}\), set
\[
  \Delta^{(11)}_S=
  \det(\omega_{11}^{ij})_{i,j\in S}.
\]
Its reduction is \(\det(u^{ij})_{i,j\in S}\).  All possible choices of
the displayed identification are related by Frobenius, so they preserve
nonvanishing.  Since \((13)\) is the unique prime above \(13\),
\[
  \det(u^{ij})_{i,j\in S}\neq0
  \quad\Longleftrightarrow\quad
  13\nmid N_{L_{11}/\mathbb Q}\bigl(\Delta^{(11)}_S\bigr).
\]
\end{proof}

\begin{proposition}\label{prop:reduction143}
If
\[
  \det W_{11}[S,S]\neq0
  \qquad\text{for every }S\subseteq\mathbb Z_{11},
\]
then every \(11\)-principal minor of \(\F_{143}\) is nonzero.
\end{proposition}

\begin{proof}
Lemma~\ref{lem:field143} converts the assumed finite-field nonvanishing
into the norm hypothesis of Proposition~\ref{prop:lifting} for every
\(S\subseteq\mathbb Z_{11}\).  That proposition applies with
\((N,p,N')=(143,13,11)\).
\end{proof}

\section{The exact finite computations}

\subsection{The characteristic-7 certificate}

\begin{proposition}\label{prop:finite70}
Every principal minor of \(W=(t^{ij})_{0\leq i,j<10}\) over
\(K=\mathbb F_{7^4}\) is nonzero.
\end{proposition}

\begin{proof}
Every field element is stored uniquely as
\[
  a_0+a_1t+a_2t^2+a_3t^3,
  \qquad a_j\in\mathbb F_7,
\]
and multiplication is reduced using
\[
  t^4=t^3-t^2+t-1.
\]
The verifier first checks irreducibility by the Rabin criterion,
\[
  X^{7^4}\equiv X\pmod{\Phi_{10}},
  \qquad
  \gcd\!\left(\Phi_{10},X^{7^2}-X\right)=1,
\]
and checks
\[
  t^{10}=1,\qquad t^5=-1,\qquad t^2\neq1.
\]

For each \(S\subseteq\{0,\ldots,9\}\), the determinant of \(W[S,S]\)
is computed by a division-free Laplace recurrence.  More explicitly,
for a \(k\times k\) matrix \(B\), and for a set \(C\) of \(r\) column
positions, let \(E(C)\) be the determinant formed by the first \(r\)
rows and the columns in \(C\), in increasing order.  With
\(E(\varnothing)=1\), expansion along the last row gives
\[
  E(C)=
  \sum_{j\in C}
  (-1)^{r-1+\operatorname{pos}_{C}(j)}
  B_{r-1,j}\,E(C\setminus\{j\}),
  \tag{3.1}\label{eq:laplace}
\]
where \(\operatorname{pos}_{C}(j)\) is the zero-based position of
\(j\) in the increasing ordering of \(C\).  Thus every arithmetic
operation takes place exactly in \(K\).

The program visits the \(1024\) masks in increasing order, including
the empty set, whose determinant is \(1\).  The exact output, grouped
by \(|S|\), is shown in Table~\ref{tab:output70}.  In the last column,
the product is taken over all subsets of the indicated size.

\begin{table}[ht]
\centering
\caption{Exact characteristic-\(7\) determinant output.}
\label{tab:output70}
\begin{tabular}{@{}rrrrl@{}}
\toprule
\(|S|\) & subsets & zero determinants & distinct values
  & product in \(K\) \\
\midrule
0  & 1   & 0 & 1  & \(1\) \\
1  & 10  & 0 & 6  & \(6\) \\
2  & 45  & 0 & 12 & \(5+2t^2+5t^3\) \\
3  & 120 & 0 & 32 & \(5\) \\
4  & 210 & 0 & 37 & \(2\) \\
5  & 252 & 0 & 54 & \(3\) \\
6  & 210 & 0 & 37 & \(2\) \\
7  & 120 & 0 & 32 & \(5\) \\
8  & 45  & 0 & 12 & \(5+2t^2+5t^3\) \\
9  & 10  & 0 & 6  & \(3\) \\
10 & 1   & 0 & 1  & \(2\) \\
\bottomrule
\end{tabular}
\end{table}

The counts sum to
\[
  \sum_{r=0}^{10}\binom{10}{r}=2^{10}.
\]
As an aggregate check, the product of all determinants is
\[
  \prod_{S\subseteq\Z_{10}}\det W[S,S]
  =6+5t^2+2t^3\neq0.
\]
Since \(K\) is a field, this also certifies that no factor vanishes.
The canonical determinant list has SHA-256 digest
\[
\begin{split}
&\texttt{dce2a62ae1c43cd52c70b9131231736c}\\[-2pt]
&\texttt{90fc23890a33c0826cf8daa100666f87}.
\end{split}
\]
The record format defining this digest is documented in the verifier.
\end{proof}

\subsection{The characteristic-13 certificate}

\begin{proposition}\label{prop:finite143}
Every principal minor of \(W_{11}=(u^{ij})_{0\leq i,j<11}\) over
\(K_{13}=\mathbb F_{13^{10}}\) is nonzero.
\end{proposition}

\begin{proof}
Every field element is stored uniquely as
\[
  a_0+a_1u+\cdots+a_9u^9,
  \qquad a_j\in\mathbb F_{13},
\]
and multiplication is reduced using
\[
  u^{10}=-(1+u+\cdots+u^9).
\]
The verifier applies the Rabin irreducibility criterion in the form
\[
  X^{13^{10}}\equiv X\pmod{\Phi_{11}},
\]
\[
  \gcd\!\left(\Phi_{11},X^{13^5}-X\right)=1,
  \qquad
  \gcd\!\left(\Phi_{11},X^{13^2}-X\right)=1,
\]
and also checks \(u^{11}=1\) and \(u\neq1\).

For each \(S\subseteq\{0,\ldots,10\}\), the determinant of
\(W_{11}[S,S]\) is computed by the division-free recurrence
\eqref{eq:laplace}.  All \(2048\) masks are visited in increasing order,
including the empty set.  To state the rankwise products compactly, put
\[
  v=u+u^3+u^4+u^5+u^9\in K_{13}.
\]
The exact output is shown in Table~\ref{tab:output143}.

\begin{table}[ht]
\centering
\caption{Exact characteristic-\(13\) determinant output.}
\label{tab:output143}
\begin{tabular}{@{}rrrrl@{}}
\toprule
\(|S|\) & subsets & zero determinants & distinct values
  & product in \(K_{13}\) \\
\midrule
0  & 1   & 0 & 1   & \(1\) \\
1  & 11  & 0 & 6   & \(1\) \\
2  & 55  & 0 & 30  & \(6-v\) \\
3  & 165 & 0 & 60  & \(4+v\) \\
4  & 330 & 0 & 120 & \(11+10v\) \\
5  & 462 & 0 & 156 & \(3v\) \\
6  & 462 & 0 & 156 & \(2+2v\) \\
7  & 330 & 0 & 120 & \(8+11v\) \\
8  & 165 & 0 & 60  & \(12+5v\) \\
9  & 55  & 0 & 30  & \(1\) \\
10 & 11  & 0 & 6   & \(11+9v\) \\
11 & 1   & 0 & 1   & \(6+12v\) \\
\bottomrule
\end{tabular}
\end{table}

The counts sum to
\[
  \sum_{r=0}^{11}\binom{11}{r}=2^{11}.
\]
The product of all determinants is
\[
  \prod_{S\subseteq\mathbb Z_{11}}\det W_{11}[S,S]
  =-1\neq0.
\]
The canonical determinant list has SHA-256 digest
\[
\begin{split}
&\texttt{830cbf79ebbc4dc6e9fbcaf846b9a0f7}\\[-2pt]
&\texttt{0ecbdc31a05404584c4e68c6a77deeed}.
\end{split}
\]
The record format is documented in the verifier.

As an independent exact check, the \(2048\) determinant polynomials in
\(\mathbb Z[\omega_{11}]\) were also computed before reduction.  Their
norms, evaluated both as multiplication-matrix determinants and as
Sylvester resultants, agree in every case, and
\[
\begin{split}
&\prod_{S\subseteq\mathbb Z_{11}}
\left|N_{\mathbb Q(\omega_{11})/\mathbb Q}
  \left(\det(\omega_{11}^{ij})_{i,j\in S}\right)\right|\\
&\hspace{18mm}={}
3^{220}11^{28160}23^{1100}67^{220}199^{220}419^{220}.
\end{split}
\]
In particular, \(13\) is absent from the prime support, independently
confirming the finite-field nonvanishing used here.
\end{proof}

\begin{proof}[Proof of Theorem~\ref{thm:main}]
Proposition~\ref{prop:finite70} verifies the hypothesis of
Proposition~\ref{prop:reduction70}.  Hence every \(10\)-principal minor
of \(\F_{70}\) is nonzero.  Similarly,
Proposition~\ref{prop:finite143} and
Proposition~\ref{prop:reduction143} show that every \(11\)-principal
minor of \(\F_{143}\) is nonzero.  Every ordinary principal minor is
\(10\)-principal in the first case and \(11\)-principal in the second,
which proves the theorem.
\end{proof}

\begin{remark}
For either \(N\in\{70,143\}\), the same conclusion holds for the
normalized Fourier matrix, since normalization multiplies an
\(r\times r\) principal minor by the nonzero scalar \(N^{-r/2}\).
Replacing \(e^{-2\pi i/N}\) by \(e^{2\pi i/N}\) merely
complex-conjugates every determinant.
\end{remark}

\section{Reproducibility}

The two direct finite-field verifiers
\[
\begin{gathered}
  \texttt{Verify\_F70\_Principal\_Minors\_v6\_2026-08-09.py},\\
  \texttt{Verify\_F143\_Principal\_Minors\_v7\_2026-08-18.py}
\end{gathered}
\]
use only the Python standard library and require Python \(3.10\) or
later.  They are run with
\[
\begin{gathered}
  \texttt{python3 Verify\_F70\_Principal\_Minors\_v6\_2026-08-09.py},\\
  \texttt{python3 Verify\_F143\_Principal\_Minors\_v7\_2026-08-18.py}.
\end{gathered}
\]
Their source-file SHA-256 digests are, in the same order,
\[
\begin{aligned}
H_{70}&=\texttt{3401ca5adb54a48aed437e5d6fdeef9f}\\[-2pt]
&\phantom{{}={}}\texttt{5cc1a0e6000606c76edb03ff5a772754},\\[3pt]
H_{143}&=\texttt{edcc523eb1fe93c1faa5cb1164abace7}\\[-2pt]
&\phantom{{}={}}\texttt{a1e13102784e0d7fd40e0c772b662bd6}.
\end{aligned}
\]
The scripts check the two field constructions and the exact orders of
the Fourier roots, visit all \(1024\) and \(2048\) index sets, assert
the binomial counts and nonvanishing, and compute the size-by-size
outputs in Tables~\ref{tab:output70} and \ref{tab:output143} and the two
determinant-list digests.

The independent norm audit is
\[
  \texttt{Verify\_F143\_Principal\_Norms\_v7\_2026-08-18.py}.
\]
Its source-file SHA-256 is
\[
\begin{split}
&\texttt{5cace5b55cd59dc19258def4b24c30cc}\\[-2pt]
&\texttt{35a69804f153cfb4d3965f8b859a4002}.
\end{split}
\]
It computes all \(2048\) determinant polynomials and evaluates every
norm by two exact integer methods.  The SHA-256 digests of its canonical
minor-polynomial, subset-norm, and sorted-spectrum records are
\[
\begin{aligned}
H_{\mathrm{poly}}&=\texttt{e90357936304ffbad91d36a2a4dd4547}\\[-2pt]
&\phantom{{}={}}\texttt{b717db11f1e25900b10a70ad341a0933},\\[3pt]
H_{\mathrm{norm}}&=\texttt{3affb5e53d5b9893d762dc4029b2a92d}\\[-2pt]
&\phantom{{}={}}\texttt{b371289d394d3ac27423e952c218ad17},\\[3pt]
H_{\mathrm{spectrum}}&=\texttt{abe21ccea40f87c5820a745c17e81da1}\\[-2pt]
&\phantom{{}={}}\texttt{1c4366caf297083b514d049a4e4f75ae}.
\end{aligned}
\]

The ancillary programs explicitly do not implement
Proposition~\ref{prop:lifting}; that mathematical step is the cited
theorem of \cite{CarageaLeeMalikiosisPfander2025}.

\section*{Acknowledgments}

The results were obtained by assistance from GPT-5.6 Sol.  The finite-field computations
were also rerun through independent exact implementations.

\end{document}